\documentclass[11pt]{amsart}
\usepackage{amssymb,pb-diagram}
\usepackage{amsmath}
\usepackage{amsthm,amssymb,amsfonts}
\usepackage{amssymb,amscd}
\usepackage{epic,eepic,epsfig}
\usepackage{latexsym}
\usepackage{xy}
\xyoption{all}

\usepackage{amssymb,amsthm,enumerate,color,mathrsfs}

\newtheorem{Thm}{Theorem}[section]

\newtheorem{Cor}[Thm]{Corollary}
\newtheorem{Lem}[Thm]{Lemma}
\newtheorem{Prop}[Thm]{Proposition}
\newtheorem{Rem}[Thm]{Remark}

\theoremstyle{definition}

\theoremstyle{remark}

\def \cal{\mathcal}

\def\eps{\varepsilon}

\def\Ndb{\mathbb N}

\def\eps{\varepsilon}
\def\ph{\varphi}

\newcommand{\bib}{\bibitem}

\begin{document}

\title{Equivalent norms with the property $(\beta)$ of Rolewicz}


\author{S.~J. Dilworth}

\address{Department of Mathematics, University of South Carolina,
Columbia, SC 29208, USA.} \email{dilworth@math.sc.edu}

\author{D. Kutzarova}
\address{Institute of Mathematics, Bulgarian Academy of
Sciences, Sofia, Bulgaria.} \curraddr{Department of Mathematics, University of Illinois
at Urbana-Champaign, Urbana, IL 61801, USA.} \email{denka@math.uiuc.edu}

\author{G. Lancien}

\address{Universit\'{e} de Franche-Comt\'{e}, Laboratoire de Math\'{e}matiques UMR 6623,
16 route de Gray, 25030 Besan\c{c}on Cedex, FRANCE.}
\email{gilles.lancien@univ-fcomte.fr}

\author{N.~L. Randrianarivony}

\address{Department of Mathematics and Computer Science, Saint Louis University, St. Louis, MO 63103, USA.}
\email{nrandria@slu.edu}

\thanks{The first author was supported by the National Science Foundation under Grant Number DMS-1361461 and by the Workshop in Analysis and Probability at Texas A\&M  University  in 2011, 2012, and 2013.}
\thanks{The second author was partially supported by the Bulgarian National
Scientific Fund under Grant DFNI-I02/10.}
\thanks{The third author was partially supported by the Workshop in Analysis and Probability at Texas A\&M  University  in 2011.}
\thanks{The fourth author was supported by the National Science Foundation under Grant Number DMS-1301591 and by the Workshop in Analysis and Probability at Texas A\&M University  in 2012 and 2014.}


\maketitle

\begin{abstract} We extend to the non separable setting many characterizations of the Banach spaces admitting an equivalent norm with the property $(\beta)$ of Rolewicz. These characterizations involve in particular the Szlenk index and asymptotically uniformly smooth or convex norms. This allows to extend easily to the non separable case some recent results from the non linear geometry of Banach spaces.
\end{abstract}

\markboth{}{}

\section{Introduction and notation}

Banach spaces admitting an equivalent norm with the property $(\beta)$ of Rolewicz recently played an important role in the non linear geometry of Banach spaces. In the separable case it has been proved in \cite{BKL} that this class is stable under coarse Lipschitz embeddings and in \cite{DKR} that it is stable under uniform quotients.

The aim of this paper is to give a complete characterization of the Banach spaces that admit an equivalent norm with the property $(\beta)$ of Rolewicz in the non separable setting. This will in particular extend Theorem 6.3 in \cite{DKLR}. As an easy consequence we will deduce that the above mentioned non linear results are also valid in the non separable case.

\medskip

Let us first define the properties of norms that will be
considered in this paper. For a Banach space $(X,\|\ \|)$ we
denote by $B_X$ the closed unit ball of $X$ and by $S_X$ its unit
sphere.

 We start with a property introduced by S. Rolewicz in \cite{Rolewicz1987too} and  now called property $(\beta)$ of Rolewicz. For its definition, we shall use a characterization due to D. Kutzarova \cite{Kutzarova1991}. An infinite-dimensional Banach space $X$ is said to have property $(\beta)$ if for any $t\in (0,a]$, where the number $1\leq a\leq 2$ depends on the space X,  there exists $\delta>0$ such that for any $x$ in $B_X$ and any $t$-separated sequence $(x_n)_{n=1}^\infty$ in $B_X$, there exists $n\ge 1$ so that
$$\frac{\|x-x_n\|}{2}\le 1-\delta.$$
For a given $t\in (0,a]$, we denote $\overline{\beta}_X(t)$ the supremum of all $\delta\ge 0$ so that the above property is satisfied. It is important to recall that a Banach space with property $(\beta)$ is reflexive (see \cite{Rolewicz1987too}).

The next definitions are due to V. Milman \cite{Milman1971} and we follow the notation from \cite{JohnsonLindenstraussPreissSchechtman2002}. For $t>0$, $x\in S_X$ and $Y$ a closed linear subspace of
$X$, we define
$$\overline{\rho}(t,x,Y)=\sup_{y\in S_Y}\|x+t y\|-1\ \ \ \ {\rm and}\ \ \
\ \overline{\delta}(t,x,Y)=\inf_{y\in S_Y}\|x+t y\|-1.$$ Then
$$\overline{\rho}_X(t)=\sup_{x\in S_X}\ \inf_{{\rm
dim}(X/Y)<\infty}\overline{\rho}(t,x,Y)\ \ \ \ {\rm and}\ \ \ \
\overline{\delta}_X(t)=\inf_{x\in S_X}\ \sup_{{\rm
dim}(X/Y)<\infty}\overline{\delta}(t,x,Y).$$ The norm $\|\ \|$ is said to be
{\it asymptotically uniformly smooth} (in short AUS) if
$$\lim_{t \to 0}\frac{\overline{\rho}_X(t)}{t}=0.$$
It is said to be {\it asymptotically uniformly convex} (in short
AUC) if
$$\forall t>0\ \ \ \ \overline{\delta}_X(t)>0.$$

Similarly, there is in $X^*$ a modulus of weak$^*$
asymptotic uniform convexity defined by
$$ \overline{\delta}_X^*(t)=\inf_{x^*\in S_{X^*}}\sup_{E}\inf_{y^*\in S_E}\{\|x^*+ty^*\|-1\},$$

\noindent where $E$ runs through all weak$^*$-closed subspaces of
$X^*$ of finite codimension.

\noindent The norm of $X^*$ is said to be {\it weak$^*$ uniformly asymptotically convex} (in short w$^*$-AUC) if
$$\forall t>0\ \ \ \ \overline{\delta}_X^*(t)>0.$$

\indent

The Szlenk index is a fundamental object which is, as we will recall, related to the existence of equivalent asymptotically uniformly smooth norms. Let us now
define it. Consider a real Banach space $X$ and $K$ a
weak$^*$-compact subset of $X^*$. For $\eps>0$ we let $\cal V$ be
the set of all relatively weak$^*$-open subsets $V$ of $K$ such
that the norm diameter of $V$ is less than $\eps$ and
$s_{\eps}K=K\setminus \cup\{V:V\in\cal V\}.$ Then we define
inductively $s_{\eps}^{\alpha}K$ for any ordinal $\alpha$ by
$s^{\alpha+1}_{\eps}K=s_{\eps}(s_{\eps}^{\alpha}K)$ and
$s^{\alpha}_{\eps}K={\displaystyle
\cap_{\beta<\alpha}}s_{\eps}^{\beta}K$ if $\alpha$ is a limit
ordinal. We
then define $\text{Sz}(X,\eps)$ to be the least ordinal $\alpha$
so that $s_{\eps}^{\alpha}B_{X^*}=\emptyset,$ if such an ordinal
exists. Otherwise we write $\text{Sz}(X,\eps)=\infty.$ The {\it
Szlenk index} of $X$ is finally defined by
$\text{Sz}(X)=\sup_{\eps>0}\text{Sz}(X,\eps)$. In the sequel
$\omega$ will denote the first infinite ordinal and $\omega_1$ the
first uncountable ordinal.

\section{Preliminary results}

Our first statement contains all the information on the duality between asymptotic uniform smoothness and weak$^*$-asymptotic uniform convexity and is an extension to the non separable setting of Proposition 2.6 in \cite{GKL2001}.

\begin{Prop}\label{Young} Let $X$ be a Banach space and $0<\sigma,\tau<1$.\\
\smallskip
(a) If $\overline{\rho}_X(\sigma)<\sigma\tau$, then  $\overline{\delta}_X^*(6\tau)\ge \sigma\tau$.\\
\smallskip
(b) If $\overline{\delta}_X^*(\tau)> \sigma\tau$, then $\overline{\rho}_X(\sigma)\le\sigma\tau$
\end{Prop}

Before to proceed with its proof we need an elementary lemma.

\begin{Lem}\label{codim} Let $X$ be a Banach space and $\eps>0$.\\
If $Y$ is a finite codimensional subspace of $X$, then there exists a weak$^*$-closed finite codimensional subspace $E$ of $X^*$ such that
$$\forall z^*\in E\ \ \sup_{y\in B_Y}|z^*(y)|\ge \frac{1}{2+\eps}\|z^*\|.$$
\end{Lem}

\begin{proof} We can write $Y=\bigcap_{i=1}^n{\rm Ker}\,x_i^* $, with $x^*_i\in X^*$. Let $F$ be the linear span of $\{x_1^*,..,x_n^*\}$. Note that $Y^\perp=F$ and $Y^*$ is canonically isometric to $X^*/F$. In other words, for any $x^*\in X^*,\ \sup_{y\in B_Y}|x^*(y)|=d(x^*,F)$. We now apply the standard Mazur technique as follows. Pick a $\eta$-net $\{u_1^*,..,u_k^*\}$ of $S_F$ and $u_i\in S_X$ so that $u_i^*(u_i)\ge 1-\eta$. If $\eta>0$ is chosen small enough then we can take $E=\bigcap_{j=1}^k{\rm Ker}\,u_j$.
\end{proof}

\begin{proof}[Proof of Proposition \ref{Young}] (a) Assume that $\overline{\rho}_X(\sigma)<\sigma\tau$ and fix $\eps>0$. Let $x^*\in S_{X^*}$ and choose $x\in S_X$ such that $x^*(x)\ge 1-\eps$. Then there exists a finite codimensional subspace $Y$ of $X$ such that
$$\forall y\in S_Y\ \ \|x+\sigma y\|<1+\sigma\tau.$$
We may also assume that $Y\subset {\rm Ker}(x^*)$. \\
Then, it follows from Lemma \ref{codim} that there exists a weak$^*$-closed finite codimensional subspace $E$ of $X^*$ such that
$$\forall z^*\in E\ \ \sup_{y\in B_Y}|z^*(y)|\ge \frac{1}{2+\eps}\|z^*\|\ \ {\rm and}\ \ z^*(x)=0.$$
Therefore
$$\forall y\in S_Y\ \ \forall z^*\in S_E\ \ \|x^*+6\tau z^*\|\ge \frac{1}{1+\sigma\tau}\langle x^*+6\sigma\tau z^*, x+\sigma y\rangle.$$
Taking the supremum over $y\in S_Y$, we get that
$$\forall z^*\in S_E\ \ \|x^*+6\tau z^*\|\ge \frac{1}{1+\sigma\tau}\Big(1-\eps+\frac{6\sigma\tau}{2+\eps}\Big).$$
Finally, letting $\eps$ tend to $0$, we conclude that
$$\overline{\delta}_X^*(6\tau)\ge \frac{1+3\sigma\tau}{1+\sigma\tau}-1\ge \sigma\tau.$$

\medskip

(b) Assume that $\overline{\delta}_X^*(\tau)> \sigma\tau$. We denote by $\cal E$ the set of all finite codimensional  subspaces of $X$. For $E,F\in \cal E$, we note $E\preceq F$ if $F\subseteq E$. Then $(E,\preceq)$ is a directed set. We shall prove our result by contradiction and therefore assume also that for some $x\in S_X$, $\overline{\rho}_X(\sigma,x)>\sigma\tau$. So we can pick $\rho$ such that $\sigma\tau <\rho < \overline{\rho}_X(\sigma,x)$ and $\overline{\delta}_X^*(\tau)> \rho$. Then
$$\forall E\in \cal E\ \ \exists x_E\in S_E\ \ \|x+\sigma x_E\|>1+\rho.$$
Thus, for any $E\in \cal E$, we can pick $y^*_E\in S_{X^*}$ such that $\langle x+\sigma x_E,y^*_E\rangle >1+\rho$. Using the weak$^*$-compactness of $B_{X^*}$ and the boundedness of $(y^*_E)$, we can find a subnet $(y^*_\alpha)_{\alpha \in A}$ of the net $(y^*_E)_{E\in \cal E}$, $x^*\in B_{X^*}$ and $c\ge0$ such that $(y^*_\alpha)_{\alpha \in A}$ is weak$^*$-converging to $x^*$ and $(\|x^*_\alpha\|)_{\alpha \in A}$ converges to $c$, where $x^*_\alpha=y^*_\alpha-x^*$.

\smallskip
(i) Assume first that $c<\tau$ and fix $\eps>0$.
$$\exists \alpha_0\in A\ \ \forall \alpha\ge \alpha_0\ \ \langle x_\alpha,x^*\rangle=0,\ |\langle x,x^*_\alpha\rangle|<\eps\ {\rm and}\ \|x^*_\alpha\|<c+\eps.$$
Then, for all $\alpha\ge \alpha_0$
$$\langle x+\sigma x_\alpha, x^*+x^*_\alpha \rangle = \langle x,x^*\rangle + \langle x,x^*_\alpha\rangle + \sigma \langle x_\alpha,x^*_\alpha\rangle  \le 1+\eps+\sigma(c+\eps).$$
Since $c<\tau$ and $\sigma\tau <\rho$, for $\eps$ initially chosen small enough this implies that
$$\forall \alpha\ge \alpha_0\ \ \langle x+\sigma x_\alpha, x^*+x^*_\alpha \rangle =\langle x+\sigma x_\alpha, y^*_\alpha\rangle <1+\rho,$$
which is impossible.

\smallskip
(ii) Assume now that $c\ge \tau$ and fix $\eps>0$. We will first show that $\|x^*\|\le 1-\sigma c$, thus we may assume that $x^*\neq 0$. Recall that $\overline{\delta}_X^*(\tau)> \rho$. So, there exists a weak$^*$-closed finite codimensional subspace $F$ of $X^*$ such that
$$\forall y^*\in S_F\ \  \Big\|x^*+\tau\|x^*\|\,y^*\Big\|\ge (1+\rho)\|x^*\|.$$
Since $(x^*_\alpha)_{\alpha \in A}$ is weak$^*$-converging to $0$ and $(\|x^*_\alpha\|)_{\alpha \in A}$ converges to $c$, we have that $d(x^*_\alpha,cS_E)$ tends to $0$. Therefore, we deduce
$$\exists \alpha_1\in A\ \ \forall \alpha\ge \alpha_1\ \ \Big\|x^*+\tau c^{-1}\|x^*\|\,x_\alpha^*\Big\|\ge (1+\rho)\|x^*\|-\eps.$$
Note that $\lambda=\tau c^{-1}\|x^*\|\in [0,1]$ and that we can write $$x^*+\lambda x^*_\alpha=\lambda(x^*+x^*_\alpha)+(1-\lambda)x^*.$$
Using the convexity of the norm we deduce that
$$\|x^*+\lambda x^*_\alpha\|\le \lambda\|x^*+x^*_\alpha\|+(1-\lambda)\|x^*\|=\lambda+(1-\lambda)\|x^*\|.$$
Therefore
$$(1+\rho)\|x^*\|-\eps \le \tau c^{-1}\|x^*\|+(1-\tau c^{-1}\|x^*\|)\|x^*\|.$$
Letting $\rho$ tend to $\sigma\tau$ and $\eps$ tend to $0$, we get that $$(1+\sigma\tau) \|x^*\|\le \tau c^{-1}\|x^*\|+(1-\tau c^{-1}\|x^*\|)\|x^*\|.$$
Dividing by $\|x^*\|$, we obtain
$$\sigma\tau \le \tau c^{-1}(1-\|x^*\|)\ \ {\rm and}\ \ \|x^*\|\le 1-\sigma c.$$
Pick now $\alpha_2\in A$ so that for all $\alpha\in A$, $\alpha\ge \alpha_2$, we have $\langle x_\alpha,x^*\rangle=0$, $|\langle x,x^*_\alpha\rangle|<\eps$ and $\|x^*_\alpha\|\le c+\eps$. Then
$$\forall \alpha\ge \alpha_2\ \ \langle x+\sigma x_\alpha,x^*+x^*_\alpha\rangle \le 1-\sigma c +\eps+\sigma (c+\eps)<1+\rho,$$
for $\eps$ small enough, which is again a contradiction.

\end{proof}

Proposition \ref{Young} can be rephrased in terms of the Young's duality between $\overline{\rho}_X$ and $\overline{\delta}^*_{X}$. Recall that for $f$ continuous
monotone increasing on $[0,1]$ with $f(0)=0$, its dual Young function is defined by
$$ f^*(s)=\sup\{st-f(t):\ 0\le t\le 1\}.$$
Let $C\ge 1$ and $f,g$ be continuous
monotone increasing functions on $[0,1]$ satisfying $f(0)=g(0)=0.$
We will say that $f,g$ are $C$-equivalent if $f(t)\ge g(t/C)$ and $g(t)\ge f(t/C)$ for all $t\in [0,1]$. Then we can state.

\begin{Cor}\label{Young2} There is a universal constant $C\ge 1$ such that for any Banach space $X$ we have that $(\overline{\rho}_X)^*$ and $\overline{\delta}^*_{X}$ are $C$-equivalent.
\end{Cor}

\medskip

The following result is now an immediate consequence. Its proof in the separable case can be found for instance in \cite{D}.

\begin{Cor}\label{duality} Let $X$ be a Banach space.\\
Then $\|\ \|_X$ is
AUS if and only if $\|\ \|_{X^*}$ is
weak$^*$-AUC.\\
In particular, if $X$ is reflexive, then $\|\ \|_X$ is
AUS if and only if $\|\ \|_{X^*}$ is AUC.
\end{Cor}

\indent

We now turn to the links between the Szlenk index and the existence of equivalent AUS renormings. The fundamental result on AUS renormings is due to Knaust, Odell and Schlumprecht in the separable case \cite{KOS} (see also \cite{OS}) and to Raja in the general case \cite{Ra}. It is the following.

\begin{Thm}\label{AUS} Let $X$ be a Banach space.\\
Then $X$ is AUS renormable if and only if $Sz(X)\le \omega$. In that case there exists $a,b>0$, $p\in (1,\infty)$ and and equivalent norm $|\ |$ on $X$ such that $\overline{\rho}_{|\ |}(t)\le at^p$ and $\overline{\delta}^*_{|\ |}(t)\ge bt^q$, where $q$ is the conjugate exponent of $p$.
\end{Thm}
The original result in \cite{KOS} or \cite{Ra} are about $\overline{\rho}_{|\ |}$. The estimate on $\overline{\delta}^*_{X}$ can be deduced from Corollary \ref{Young2}.

\medskip

As an immediate consequence of the two previous statements we have.

\begin{Cor}\label{AUCreflexive}
Let $X$ be a reflexive Banach space. Then the following are equivalent.\\
\smallskip
(a) $Sz(X^*)\le \omega$.\\
\smallskip
(b) There exists $a>0$, $p\in (1,\infty)$ and and equivalent norm $|\ |$ on $X^*$ such that $\overline{\rho}_{|\ |}(t)\le at^p$ (we say that $|\ |$ is AUS with power type $p$).\\
\smallskip
(c) There exists $b>0$, $q\in (1,\infty)$ and and equivalent norm $|\ |$ on $X$ such that $\overline{\delta}_{|\ |}(t)\ge bt^q$ (we say that $|\ |$ is AUC with power type $q$).
\end{Cor}

\section{Szlenk index and subspaces or quotients with a basis}

In this section we show that the Szlenk index of a reflexive Banach space is determined by its subspaces or quotients with a basis, when it is countable. For subspaces, it is a refinement of Proposition 3.1 in \cite{La1}. The result for quotients was given by Proposition 3.5 in \cite{La1}.

\begin{Prop}\label{basicdetermination} Let $X$ be a reflexive Banach space and let $\alpha <\omega_1$ so that $Sz(X)>\alpha$. Then\\
\smallskip
(i) There is a subspace $Y$ of $X$ with a basis such that $Sz(Y)>\alpha$.\\
\smallskip
(ii) There exists a subspace $Z$ of $X$ such that $X/Z$ has a basis and $Sz(X/Z)>\alpha$.
\end{Prop}

\begin{proof}\

(i) We will just slightly refine the proof of Proposition 3.1 in \cite{La1}. So let us recall the construction. We first introduce  a family
$(T_\alpha)_{\alpha < \omega_1}$ of well founded trees in $\omega^{< \omega}$ (the set of finite sequences of natural numbers) constructed inductively as follows:\\
$T_0 = \{\emptyset\}$\\
$T_{\alpha+1} = \{\emptyset\} \cup \bigcup
\limits_{n=0}^\infty n^\frown T_\alpha$, where $n^\frown
T_\alpha = \{n^\frown s, s \in T_\alpha\}$.\\
$T_\alpha = \{\emptyset\} \cup \bigcup
\limits_{n=0}^\infty n^\frown T_{\alpha_n}$, if $\alpha$
is a limit ordinal and $(\alpha_n)_{n=0}^\infty$ an enumeration of $[0,\alpha)$.\\
The trees $T_\alpha$ are equipped with their natural order: $(m_1,..,m_l)\le (n_1,..,n_k)$ if $l\le k$ and $m_i=n_i$ for $i\le l$.\\
Note that the height of $T_\alpha$ is $ht(T_\alpha) = \alpha$. Moreover, if for $s$ in $T_\alpha$ we denote $T_\alpha (s) = \{t \in
\omega^{< \omega},\ s^\frown t \in T_\alpha\}$ and $h_\alpha (s) =
ht(T_\alpha(s))$, we have that $T_\alpha (s) =
T_{{h_\alpha} (s)}$.\\
Recall also (see \cite{La1}) that for any $1 \leq \alpha < \omega_1$,
there exists a bijection $\ph_\alpha : \omega \to
T_\alpha$ such that for any $s,s'$ in $T_\alpha,~s < s'$ implies
$\ph_\alpha^{-1}(s) < \ph_\alpha^{-1}(s')$.

Assume now that $Sz(X)>\alpha$. So there exist $\eps>0$ and $x^*\in S_\eps^\alpha(B_{X^*})$. Pick a sequence $(\eps_n)_{n=1}^\infty$ such that $\prod_{n=1}^\infty (1+\eps_n)\le 2$. We build, by induction on $n$, $(x_{{\ph_\alpha}(n)}^*)_{n=0}^\infty$ in $B_{X^*}$,  $(x_n)_{n=1}^\infty$ in $S_X$, a finite $\eps_n$-net $F_n$ of the unit ball of the linear span of $\{x_1,..,x_n\}$ and $G_n$ a finite subset of $S_{X^*}$ which is norming for $F_n$ so that:\\
\smallskip
(a) $x_{{\ph_\alpha}(0)}^* = x_\emptyset^* = x^*$.\\
\smallskip
(b) $\forall n \in \omega\ \ x_{{\ph_\alpha}(n)}^* \in
s_\eps^{h_\alpha(\ph_\alpha(n))}(B_{X^*})$\\
\smallskip
(c) $\forall n \geq 1$, $x_n\in G_{n-1}^\perp$.\\
\smallskip
(d) $\forall n \geq 1\ \ (x_{{\ph_\alpha}(n)}^* - x_{s_n}^*)
(x_n) > \frac{\eps}{6}$, where $s_n$ is the predecessor of $\ph_\alpha (n)$.\\
\smallskip
(e) $\forall n \geq 2\ \ \forall 1 \leq k \leq
n-1\ \ |(x_{{\ph_\alpha}(n)}^* - x_{s_n}^*)(x_k)| \leq 2^{-n}$.
\smallskip

Set $x_{{\ph_\alpha}(0)}^* = x_\emptyset^* = x^*$. Let now $n\ge 1$ and assume that $x_{{\ph_\alpha}(k)}^*$, $x_k$, $F_k$ and $G_k$ have been constructed for $0 \leq k \leq n-1$  and satisty a)...e). There exists $i_n < n$ such that $\ph_\alpha(n) =
\ph_\alpha(i_n)^\frown k_n$, with $k_n \in \Ndb$. By induction hypothesis
$$x_{{\ph_\alpha}(i_n)}^* \in s_\eps^{h_\alpha(\ph_\alpha(i_n))}(B_{X^*})\subset s_\eps^{h_\alpha(\ph_\alpha(n))+1}(B_{X^*}).$$
Since $X$ is reflexive, we can apply Lemma \ref{codim} and the above inclusion to find $x_{{\ph_\alpha}(n)}^*$ and $x_n\in S_{G_{n-1}^\perp}$ such that (d) and (e) are satisfied. Then we pick an $\eps_n$-net $F_n$ of the unit ball of the linear span of $\{x_1,..,x_n\}$ and $G_n$ a finite subset of $S_{X^*}$ which is 1-norming for $F_n$.  Our inductive construction  is now finished.

Note that, by the standard Mazur technique, our condition (c)  implies that $(x_n)_{n=1}^\infty$ is a 2-basic sequence. Let $Y$ be the closed linear span of $\{x_n,\ n\ge 1\}$. For $s\in T_\alpha$, we denote $y^*_s$ the restriction of $x^*_s$ to $Y$. Condition (e) implies that for any $s\in T_\alpha'$ (the set of sequences in $T_\alpha$ having a successor in $T_\alpha$), the sequence $(y^*_{s\frown n})_n$ is weak$^*$-converging to $y^*_s$ in $Y^*$. Condition (d) insures that for all $s\in T_\alpha'$ and all $n\in \Ndb$, $\|y_{s\frown n}^*-y^*_s\|_{Y^*}\ge \frac{\eps}{6}$. We deduce that $y^*_\emptyset \in s_{\eps/6}^\alpha(B_{Y^*})$ and that $Sz(Y)>\alpha$.

\medskip (ii) This is a particular case of Proposition 3.5 in \cite{La1}. The reflexive case being easier, let us just point for the sake of completeness that in the above construction, we only pick $x_n\in S_X$ such that $(x_{{\ph_\alpha}(n)}^* - x_{s_n}^*)
(x_n) > \frac{\eps}{3}$. If we denote $u_n^*=x_{{\ph_\alpha}(n)}^* - x_{s_n}^*$, still using the standard Mazur gliding hump technique, we can insure that $(u_n^*)$ is a basic sequence. Since $X$ is reflexive, we have that $(u_n^*)$ converges weakly to 0 in $X^*$. Let now $F$ be the closed linear span of $(x_{{\ph_\alpha}(n)}^*)_{n=0}^\infty$ and $Z=F^\perp \subset X$. We still have that $(u_n^*)$ converges to 0 weakly in $F$, i.e. weak$^*$ in $(X/Z)^*$. It then follows from the construction of $(u_n^*)$ that $x^*\in s_{\frac{\eps}{3}}^\alpha (B_{(X/Z)^*})$ and $Sz(X/Z)>\alpha$.

\end{proof}

\section{Main result}

We start this section by recalling the know results on reflexive spaces having an equivalent norm which is AUS and AUC. It was shown in \cite{Ku}, with a different terminology, that if a norm of a reflexive Banach space is AUS and AUC then it has property $(\beta)$. We will need the following quantitative version of this result (see Theorem 5.2 in \cite{DKRRZ}).

\begin{Thm}\label{betapower} Let $X$ be a reflexive Banach space. Assume that $\overline{\rho}_{X}(t)\le at^p$ and $\overline{\delta}_{X}(t)\ge bt^q$ for some $a,b>0$ and $p,q \in (1,\infty)$ and all $t\le 1$, then there exits $c>0$ such that $\overline{\beta}_{X}(t)\ge ct^{\frac{pq-p}{p-1}}$, for $t\le 1$.
\end{Thm}

Note that Theorem 5.2 in \cite{DKRRZ} is stated with  indices called  $b_X$ and $d_X$  instead of  $\overline{\rho}_{X}$ and $\overline{\delta}_{X}$ respectively, where
$$b_X(t)=\sup\{\limsup_{n\to \infty} \|x+tx_n\|-1\}\ \ {\rm and} \ \ d_X(t)=\inf\{\liminf_{n\to \infty} \|x+tx_n\|-1\},$$
the above $\sup$ and $\inf$ being taken over all weakly null sequences in $B_X$. The proof of Theorem 5.2 in \cite{DKRRZ} can be easily adapted to the moduli $\overline{\rho}_{X}$ and $\overline{\delta}_{X}$.

\medskip

We will now prove our main result, which gathers all the linear characterizations of the Banach spaces admitting an equivalent norm with property $(\beta)$. Before to state it, let us recall that a blocking of a basic sequence $(e_n)_{n=1}^\infty$ in a Banach space $X$ is a sequence of subspaces $(X_k)_{k=1}^\infty$ of $X$ such that for any $k\ge 1$, $X_k$ is the linear span of $\{e_{n_k},..,e_{n_{k+1}-1}\}$, where $n_1=1$ and $(n_k)_{k=1}^\infty$ is an increasing sequence in $\Ndb$. For $C\ge 1$ and $p,q\in (1,\infty)$, we say that such a blocking satisfies $(C,p,q)$ estimates if for all $x_1,..,x_m$ in the linear span of $(e_n)$ with consecutive disjoint supports with respect to the blocking $(X_k)$:
$$\frac{1}{C}\big(\sum_{i=1}^m \|x_i\|^p\big)^{1/p} \le \big\|\sum_{i=1}^m x_i\big\| \le C \big(\sum_{i=1}^m \|x_i\|^q\big)^{1/q}.$$

\begin{Thm}\label{main} Let $X$ be a reflexive Banach space. The following assertions are equivalent.

\smallskip\noindent
(i) $Sz(X)\le \omega$ and $Sz(X^*)\le \omega$.

\smallskip\noindent
(ii) There exists $p\in (1,\infty)$ such that $X$ admits an equivalent AUS norm with power type $p$ and there exists $q\in (1,\infty)$ such that $X$ admits an equivalent AUC norm with power type $q$.

\smallskip\noindent
(iii) There exist $p,q\in (1,\infty)$ such that $X$ admits an equivalent norm which is simultaneously AUS of power type $p$ and AUC of power type $q$.

\smallskip\noindent
(iv) There exists $r\in (1,\infty)$ such that $X$ admits an equivalent norm with property $(\beta)$ with power type $r$.

\smallskip\noindent
(v) There exists $C\ge 1$ and $1<q\le p<\infty$ so that each basic sequence in $X$ has a blocking which admits $(C,p,q)$ estimates.
\end{Thm}

The main implication to prove is $(ii) \Rightarrow (iii)$. Let us first state and prove separately a more precise statement.

\begin{Lem}\label{powertypes} Let $X$ be a reflexive Banach space. Assume that there exist $p,q\in (1,\infty)$ such that $X$ admits an equivalent AUS norm with power type $p$ and an equivalent AUC norm with power type $q$. Then $X$ admits an equivalent norm which is simultaneously AUS of power type $p$ and AUC of power type $4q$.
\end{Lem}

\begin{proof} We shall adapt the proof of Proposition IV.5.2 in \cite{DGZ}, which is due to John and Zizler \cite{JZ}. The technique is a variant of the so-called Asplund averaging method initiated by E. Asplund in \cite{A}. So assume that there exist $a,b>0$ such that $X$ admits an equivalent norm $N$ satisfying $\overline\delta_N(t)\ge at^q$ for all $t\in (0,1)$ and an equivalent norm $M$ satisfying $\overline\rho_M(t)\le bt^p$ for all $t\in (0,1)$. Assume also, as we may, that for all $x\in X$, $N(x)\le M(x)\le CN(x)$, with $C\ge 1$. Note that Corollary \ref{duality} insures the existence of $c>0$ so that $\overline\delta_{M^*}(t)\ge ct^{q'}$, where $q'$ is the conjugate exponent of $p$. We now define
$$\forall n\ge 1 \ \ \forall x^*\in X^*\ \ \|x^*\|^*_n=N^*(x^*)+n^{-1}M^*(x^*).$$
Then $\|\ \|_n$ is the predual norm of $\|\ \|^*_n$ on $X$ and we set
$$\forall x\in X\ \ |x|=\sum_{n=1}^\infty n^{-3} \|x\|_n.$$

We start with the estimation of $\overline\rho_{|\ |}$. It is easy to check that there exists $c_1>0$ such that for all $n\ge 1$, $\overline\delta_{\|\ \|_n^*}(t)\ge c_1n^{-1}t^{q'}$. Then Corollary \ref{duality} yields the  existence of $b_1>0$ such that $\overline\rho_{\|\ \|_n}(t)\le 2nb_1t^p$.\\
Let now $x\in X$ so that $|x|=1$ and $t\in (0,1)$. Note that there exist $A,B>0$ such that for all $x\in X$ and $n\ge 1$, $A|x|\le \|x\|_n \le B|x|$ and $A|x|\le N(x) \le B|x|$. Therefore $|y|\le t$ implies that $\|y\|_n\|x\|_n^{-1}\le BA^{-1}t$.\\
Pick $n_0\ge 1$ such that $2B\sum_{n>n_0}n^{-3}\le Ct^p$, with $C>0$ that will be made precise later. Then we can find a finite codimensional subspace $E$ of $X$ such that for all $y\in E$ with $|y|\le t$ we have
$$\forall n\le n_0\ \ \|x+y\|_n= \|x\|_n\Big\|\frac{x}{\|x\|_n}+\frac{y}{\|x\|_n}\Big\|_n\le \|x\|_n(1+nb_1B^pA^{-p}t^p)$$
$$\le \|x\|_n+nb_1B^{p+1}A^{-p}t^p.$$
Therefore
$$\sum_{n=1}^{n_0} n^{-3}\|x+y\|_n \le \sum_{n=1}^{n_0} n^{-3}\|x\|_n + t^pb_1B^{p+1}A^{-p}\sum_{n=1}^{n_0} n^{-2}\le 1+Ct^p,$$
where $C=b_1B^{p+1}A^{-p}\sum_{n=1}^\infty n^{-2}$. Note that $\|x+y\|_n\le B(1+t)\le 2B$. Finally, our initial choice of $n_0$ implies that $|x+y|\le 1+2Ct^p$. We have shown that $|\ |$ is AUS with power type $p$.

Let us now turn to the study of $\overline\delta_{|\ |}$. First, note that
$$\forall n\ge 1\ \ \forall x\in X\ \ \big(1-\frac1n\big)N(x)\le \|x\|_n \le N(x).$$
Consider $x\in X$ so that $|x|=1$ and $t\in (0,1)$. Let now $n_0$ be the integer so that $n_0-1\le \alpha t^{-q}\le n_0$, where the choice of the constant $\alpha>0$ will be explained later.  Recall that $|y|\ge t$ implies that for all $n\ge 1$, $\|y\|_n\|x\|_n^{-1}\ge AB^{-1}t$ and $N(y)N(x)^{-1}\ge AB^{-1}t$. So we may pick a finite codimensional subspace $E$ of $X$ such that for all $y\in B_E$, with $|y|\ge t$ we have
\begin{equation}\label{eq1}
N(x+y)\ge N(x)\big(1+\frac{aA^q}{B^q}t^q\big).
\end{equation}
Thus
\begin{equation}\label{eq2}
\|x+y\|_{n_0} = \big(1-\frac{1}{n_0}\big)N(x+y)\ge \big(1-\frac{t^q}{\alpha}\big)\big(1+\frac{aA^q}{B^q}t^q\big)\|x\|_{n_0}
\ge \|x\|_{n_0}+ Dt^q,
\end{equation}
with $D=aA^{q}(2B^q)^{-1}$, if $\alpha$ was initially chosen large enough (depending only on $a,A,B$).\\
We have also that for all $n>n_0$ and all $y\in B_E$ with $|y|\ge t$,
\begin{equation}\label{eq3}
\|x+y\|_{n}\ge (1-\frac1n)N(x+y)\ge \big(1-\frac{t^q}{\alpha}\big)\big(1+\frac{aA^q}{B^q}t^q\big)\|x\|_{n}\ge \|x\|_n.
\end{equation}
Given $(\eps_n)_{n=1}^{n_0-1}$ in $(0,1)$, the standard Mazur technique allows us to choose $E$ such that we also have
$$\forall n<n_0\ \ \forall y\in B_E\ \ \|x+y\|_n\ge \|x\|_n-\eps_n.$$
Therefore, if the $\eps_n$'s are chosen small enough, we get that
\begin{equation}\label{eq4}
\forall n<n_0\ \ \forall y\in B_E\ \ \sum_{n=1}^{n_0-1}n^{-3}\|x+y\|_n \ge \sum_{n=1}^{n_0-1}n^{-3}\|x\|_n-\frac{D}{2n_0^3}t^q.
\end{equation}
Note that $n_0\le \alpha t^{-q}+1 \le t^{-q}(\alpha+1)$. Then, summing equations (\ref{eq2}), (\ref{eq3}) and (\ref{eq4}) we obtain that for all $y\in B_E$ with $|y|\ge t$
$$|x+y|\ge |x|+\frac{D}{2n_0^3}t^q\ge 1+\frac{D}{2(\alpha+1)^3}t^{4q},$$
which concludes our proof.
\end{proof}

\begin{Rem} It follows from Corollary \ref{duality} that we could also build an equivalent norm which is AUC of power type $q$ and AUS with a power type in $(1,p)$. We do not know if it is possible to preserve both power types together. To the best of our knowledge the similar question for uniformly smooth and convex norms is still open.
\end{Rem}

\begin{Rem} It is worth mentioning that one can use a nice Baire category argument to show that if a reflexive Banach space $X$ admits an equivalent AUS norm and an equivalent AUC norm then it admits an equivalent norm which is simultaneously AUS and AUC. However, the very nature of Baire's Lemma makes it quite hopeless to get a quantitative result such as Lemma \ref{powertypes} so easily. Let us nevertheless describe this qualitative argument.
\end{Rem}

\begin{proof} We shall adopt here  the presentation of the Asplund averaging technique, based on Baire's Lemma that can be found in \cite{DGZ} (p. 52-56). Let us denote by $P$ the set of norms on $X$ that are equivalent to the original norm  $\|\ \|$ and let $B$ be the closed unit ball of $\|\ \|$. For $N,M\in P$, we set
$$d(N,M)=\sup\{|N(x)-M(x)|,\ x\in B\}.$$
Then $(P,d)$ is open in the complete metric space of all continuous semi-norms on $(X,\|\ \|)$ equipped with $d$ and therefore a Baire space. \\
Assume first that $\|\ \|$ is AUC. We will the show that $\cal C=\{N\in P,\ N\ {\rm is}\ {\rm AUC}\}$ is a dense $\cal G_\delta$ subset of $P$. For $n\in \Ndb$, we denote $O_n$ the set of all $N$ in $P$ so that there exists $\delta>0$ such that forall $x\in X$ with $N(x)=1$, there exists a finite codimensional subspace $Y$ of $X$ satisfying:
$$\forall y\in Y\ {\rm with}\ N(y)=1,\ N(x+2^{-n}y)>1+\delta.$$
It is not difficult to check that for all $n\in \Ndb$, $O_n$ is an open subset of $P$ and that $\cal C=\bigcap_{n\in \Ndb}O_n$. This shows that $\cal C$ is a $\cal G_\delta$ subset of $P$.\\
It is also an exercise to verify that for any $N\in P$ and any $n\in \Ndb$, the norm $N+2^{-n}\|\ \|$ is AUC. This finishes the proof of the fact that $\cal C$ is a dense $\cal G_\delta$ subset of $P$, whenever it is non empty.\\
Following \cite{DGZ}, we now denote $P^*$ the set of all norm on $X^*$ that are dual to an element of $P$. If $B^*$ is the closed unit ball of the dual norm of $\|\ \|$, we set
$$\forall N^*,M^*\in P^*\ \ d^*(N^*,M^*)=\sup\{|N^*(x^*)-M^*(x^*)|,\ x^*\in B^*\}.$$
It is shown in \cite{DGZ} that the map $N\mapsto N^*$, where $N^*$ is the dual norm of $N$, is a homeomorphism from $(P,d)$ onto $(P^*,d^*)$. Therefore, it follows from our study of $\cal C$ and from Corollary \ref{duality} that the set $\cal S$ of all AUS norms in $P$ is a dense $\cal G_\delta$ subset of $P$, whenever it is non empty.\\
Finally, we use that $(P,d)$ is a Baire space to obtain that $\cal C \cap \cal S$ is non empty whenever $\cal C$ and $\cal S$ are non empty.
\end{proof}

Let us now gather all these results to complete the proof of our main result.

\begin{proof}[Proof of Theorem \ref{main}] The equivalence $(i) \Leftrightarrow (ii)$  follows from Theorem \ref{AUS} and Corollary \ref{AUCreflexive}.

\medskip\noindent $(ii) \Rightarrow (iii)$ This follows from Lemma \ref{powertypes}.

\medskip\noindent $(iii) \Rightarrow (iv)$ This is a direct consequence of Theorem \ref{betapower}.

\medskip\noindent $(iv) \Rightarrow (i)$ Assume that the norm of $X$ has property $(\beta)$. Then, so does every separable subspace or separable quotient of $X$. Thus every separable subspace or quotient of $X$ has a Szlenk index not exceeding $\omega$ (see Theorem 6.3 in \cite{DKLR} and references therein). We can now apply Proposition \ref{basicdetermination}  to deduce that $Sz(X)\le \omega$ and $Sz(X^*)\le \omega$.

\medskip\noindent $(v) \Rightarrow (i)$ Assuming $(v)$, we deduce from Theorem 4.3 in \cite{Pr} (see also Corollary 9 in \cite{Ku}) that every subspace $Y$ of $X$ with a basis admits an equivalent norm which is AUS and AUC and therefore is such that $Sz(Y)\le \omega$ and $Sz(Y^*)\le \omega$. Then, it follows from Proposition \ref{basicdetermination} that  $Sz(X)\le \omega$ and $Sz(X^*)\le \omega$.

\medskip\noindent $(iii)\Rightarrow (v)$ Assume that $\overline\rho_X(t)\le at^p$ and $\overline\delta_X(t)\ge bt^q$  for some $a,b>0$ and $p,q\in (1,\infty)$. Let $(e_n)$ be a basic sequence in $X$ and $Y$ be the closed linear span of the $e_n$'s. Therefore  $\overline\rho_Y(t)\le \overline\rho_X(t) \le at^p$ and $\overline\delta_Y(t)\ge \overline\delta_X(t)\ge bt^q$. Then it follows from Theorem 4.3 in \cite{Pr} that there exists $C\ge 1$ and $1<s\le r<\infty$ so that $(e_n)$ has a blocking which admits $(C,r,s)$ estimates.
A careful reading of the proof shows moreover that $C,r$ and $s$ are controlled by $a,b,p$ and $q$ and therefore can be chosen independently of the basic sequence $(e_n)$.

\end{proof}

\section{Applications to the non linear geometry of Banach spaces.}

Let us first recall the definition of the countably branching hyperbolic tree. For a positive integer $N$, We denote
$T=\bigcup_{i=0}^\infty \Ndb^i$, where $\Ndb^0:=\{\emptyset\}$, the set of all finite sequences of positive integers. For $s\in T$, we denote by $|s|$ the length of $s$.
There is a natural ordering on $T$ defined by $s\le t$ if $t$ extends
$s$. If $s\le t$, we will say that $s$ is an ancestor of $t$. Then we equip
$T$ with the hyperbolic distance $\rho$, which is defined as follows. Let $s$ and $s'$ be two elements of $T$ and let $u\in T$ be their greatest common ancestor. We set
$$\rho(s,s')=|s|+|s'|-2|u|=\rho(s,u)+\rho(s',u).$$
Let $(M,d)$, $(N,\delta)$ be metric spaces. A map $f:M\to N$ is a bi-Lipschitz embedding if there exist constants $A,B>0$ such that
$$\forall x,y\in M\ \ Ad(x,y)\le \delta(f(x),f(y))\le Bd(x,y).$$
We say that $f$ is a coarse Lipschitz embedding if there exist constants $\theta,A,B>0$ such that
$$\forall x,y\in M\ \ d(x,y)\ge \theta \Rightarrow Ad(x,y)\le \delta(f(x),f(y))\le Bd(x,y).$$

It has been proved in \cite{BKL} that for a reflexive Banach space $X$, $Sz(X)\le \omega$ and $Sz(X^*)\le \omega$ if and only if there is no bi-Lipschitz embedding from $(T,\rho)$ into $X$. We also refer to \cite{BZ} for a short, elegant and quantitative proof of the fact that there is no bi-Lipschitz embedding from $(T,\rho)$ into a Banach space with property $(\beta)$.

\medskip

We can now state an easy generalization of Theorem 4.3 in \cite{BKL}

\begin{Cor} Let $Y$ be a Banach space with property $(\beta)$ and assume that a Banach space $X$ coarse Lipschitz embeds into $Y$. Then $X$ admits an equivalent norm with property $(\beta)$.
\end{Cor}

\begin{proof} We have that any separable subspace $Z$ of $X$ coarse Lipschitz embeds into $Y$ and therefore into a separable subspace of $Y$ with Szlenk index at most $\omega$. Theorem 4.1 in \cite{BKL} insures that $Z$ is reflexive. Since the reflexivity is separably determined, we deduce that $X$ is reflexive. Assume now that $X$ does not admit an equivalent norm with property $(\beta)$. Then it follows from Theorem \ref{main} that $Sz(X)> \omega$ or $Sz(X^*)> \omega$. We now use Proposition \ref{basicdetermination} to deduce the existence of a separable subspace $Z$ of $X$ such that $Sz(Z)> \omega$ or $Sz(Z^*)> \omega$. Then Corollary 3.4 in \cite{BKL} implies that there exists a bi-Lipschitz embedding $f$ from $(T,\rho)$ into $Z$ and therefore into a separable subspace $E$ of $Y$ (using the fact that $T$ is countable, uniformly discrete and a dilation if necessary). This contradicts the fact that $E$ has property $(\beta)$. To justify this last statement, one could use Corollary 3.4 in \cite{BKL} and Theorem \ref{main}, but it is more natural to use directly Theorem 2.1 in \cite{BZ}.
\end{proof}

We recall that a map $f:M\to N$  is called co-uniformly continuous if for every $\eps >0$, there exists $\delta>0$ such that for every $x\in X$,
$$B(f(x), \delta) \subset f\big(\left(B(x,\eps)\right)\big),$$
where $B(x,\eps)$ denotes the closed ball of center $x$ and radius $\eps$.\\
A map $f:M\to N$ that is both uniformly continuous and co-uniformly continuous  is called a uniform quotient. We can now extend to the non separable setting the following result from \cite{DKR}.

\begin{Thm} Let $X,Y$ be two Banach spaces.\\
Assume that $X$ has property $(\beta)$ and that $f:X\to Y$ is a uniform quotient. Then $Y$ has an equivalent norm with property $(\beta)$.
\end{Thm}
\begin{proof} Under our assumptions, any separable linear quotient of $Y$ is a uniform quotient of $X$. It follows from \cite{DKR} that any separable linear quotient of $Y$ has an equivalent norm with property $(\beta)$. First, since reflexivity is determined by separable quotients and is implied by property $(\beta)$, we obtain that $Y$ is reflexive. Then we can apply Proposition \ref{basicdetermination} together with Theorem \ref{main} to deduce that $Y$ has an equivalent norm with property $(\beta)$.

\end{proof}

\end{document}